\documentclass{article} \UseRawInputEncoding
\usepackage{amsmath,amsthm,amssymb,amsfonts}
\usepackage{verbatim}
\usepackage{graphicx}
\usepackage{stmaryrd} 
\usepackage{hyperref}

\usepackage[colorinlistoftodos]{todonotes}

\theoremstyle{plain}
\newtheorem{thm}{Theorem}

\newtheorem{prop}[thm]{Proposition}

\newtheorem{remark}[thm]{Remark}

\theoremstyle{definition}
\newtheorem{definition}[thm]{Definition}

\numberwithin{thm}{section}

\newcommand{\adj}{\leftrightarrow}
\newcommand{\adjeq}{\leftrightarroweq}

\def\Z{{\mathbb Z}}
\def\N{{\mathbb N}}

\title{Variants on Digital Covering Maps}
\author{Laurence Boxer
\thanks{
    Department of Computer and Information Sciences,
    Niagara University,
    Niagara University, NY 14109, USA;
    and Department of Computer Science and Engineering,
    State University of New York at Buffalo.
    email: boxer@niagara.edu
}
}

\date{ }
\begin{document}
\maketitle{}

\begin{abstract}
S-E Han's paper~\cite{HanEquiv} discusses several
variants of digital covering maps. We show 
several equivalences among these variants
and discuss shortcomings in Han's paper.


\end{abstract}

\maketitle

\section{Introduction}
The notion of a covering map has been adapted from
classical algebraic topology to digital topology,
where it is an important tool for computing
digital versions of fundamental groups for binary
digital images. With varying success, attempts have
been made to modify the notion of a digital covering
map to obtain related results under less restrictive
conditions. Among these attempts are Han's
paper~\cite{HanEquiv}, which contains 
a proof that is murky (see section~\ref{unmurk}
for clarification) and citations that are
inappropriate. We also discuss a strangely 
presented example in Han's related 
paper~\cite{HanUnique}
(see Remark~\ref{badExl}). Further, 
it turns out that
some of Han's variants on covering maps don't
really vary from covering maps
(see Theorem~\ref{coverEquivs}).
Also, some of material of~\cite{HanEquiv}
is superseded by other 
papers including~\cite{BxHtpProps,Pak22,PakZak}. 
We justify these claims in the current paper.

\section{Preliminaries}
We use $\N$ for the set of natural numbers,
$\Z$ for the set of integers, and
$\#X$ for the number of distinct members of $X$.

We typically denote a (binary) digital image
as $(X,\kappa)$, where $X \subset \Z^n$ for some
$n \in \N$ and $\kappa$ represents an adjacency
relation of pairs of points in $X$. Thus,
$(X,\kappa)$ is a graph, in which members of $X$ may be
thought of as black points, and members of $\Z^n \setminus X$
as white points, of a picture of some ``real world" 
object or scene.

\subsection{Adjacencies}
Let $u,n \in \N$, $1 \le u \le n$. 
Han's papers use ``$k$-adjacency" sometimes to mean
an arbitrary adjacency, sometimes as an abbreviation
for what he calls ``$k(u,n)$-adjacency," where the
digital image $(X,k)$ satisfies $X \subset \Z^n$ and
$x = (x_1, \ldots, x_n),~y=(y_1, \ldots, y_n) \in X$ 
are $k(u,n)$-adjacent if and only if
\begin{itemize}
    \item $x \neq y$, and
    \item for at most $u$ indices~$i$, 
          $\mid x_i - y_i \mid = 1$, and
    \item for all indices $j$ such that 
          $\mid x_j - y_j \mid \neq 1$, we have
          $x_j = y_j$.
\end{itemize}
Other authors refer to this adjacency as $c_u$-adjacency.
We will prefer the latter notation in the current paper.
The $c_u$ adjacencies are the adjacencies most used
in digital topology, especially $c_1$ and $c_n$.

In low dimensions, it is also common to denote a
$c_u$ adjacency by the number of points that can
have this adjacency with a given point in $\Z^n$. E.g.,
\begin{itemize}
    \item For subsets of $\Z^1$, $c_1$-adjacency is 2-adjacency.
    \item For subsets of $\Z^2$, $c_1$-adjacency is 4-adjacency and
          $c_2$-adjacency is 8-adjacency.
    \item For subsets of $\Z^3$, $c_1$-adjacency is 6-adjacency,
          $c_2$-adjacency is 18-adjacency, and
          $c_3$-adjacency is 26-adjacency.
\end{itemize}

We use the notations $y \adj_{\kappa} x$, or, when
the adjacency $\kappa$ can be assumed, $y \adj x$, to mean
$x$ and $y$ are $\kappa$-adjacent.
The notations $y \adjeq_{\kappa} x$, or, when
$\kappa$ can be assumed, $y \adjeq x$, mean either
$y=x$ or $y \adj_{\kappa} x$.
For $x \in X$, let
\[ N(X,x,\kappa) =
  \{ \, y \in X \mid x \adj_{\kappa} y \, \}.
\]
When the image $(X,\kappa)$ under discussion is clear, we
will use the notations $N(x)$ or $N_{\kappa}(x)$
as follows.
\[
N(x) = \{ \, y \in X \mid y \adjeq_{\kappa} x \, \} =
N(X,x,\kappa) \cup \{x\}.
\]

A sequence $P=\{y_i\}_{i=0}^m$ in a digital image $(X,\kappa)$ is
a {\em $\kappa$-path from $a \in X$ to $b \in X$} if
$a=y_0$, $b=y_m$, and $y_i \adjeq_{\kappa} y_{i+1}$ 
for $0 \leq i < m$.

$X$ is {\em $\kappa$-connected}~\cite{Rosenfeld},
or {\em connected} when $\kappa$
is understood, if for every pair of points $a,b \in X$ there
exists a $\kappa$-path in $X$ from $a$ to $b$.

A {\em (digital) $\kappa$-closed curve} is a
path $S=\{s_i\}_{i=0}^m$ such that $s_0=s_m$,
and $0 < |i - j| < m$ 
implies $s_i \neq s_j$. If, also, $0 \le i < m$ implies
\[ N(S,x_i,\kappa)=\{x_{(i-1)\mod n},~x_{(i+1)\mod m}\}
\]
then $S$ is a {\em (digital) 
$\kappa$-simple closed curve}.

\subsection{Digitally continuous functions}
Digital continuity is defined
to preserve connectedness, as at
Definition~\ref{continuous} below. By
using adjacency as our standard of ``closeness," we
get Theorem~\ref{continuityPreserveAdj} below.

\begin{definition}
\label{continuous}
{\rm ~\cite{Boxer99} (generalizing a definition of~\cite{Rosenfeld})}
Let $(X,\kappa)$ and $(Y,\lambda)$ be digital images.
A function $f: X \rightarrow Y$ is 
{\em $(\kappa,\lambda)$-continuous} if for
every $\kappa$-connected $A \subset X$ we have that
$f(A)$ is a $\lambda$-connected subset of $Y$.
\end{definition}

If either of $X$ or $Y$ is a subset of the 
other, we use the abbreviation
$\kappa$-continuous for $(\kappa,\kappa)$-continuous.

When the adjacency relations are understood, we will simply say that $f$ is \emph{continuous}. Continuity can be expressed in terms of adjacency of points:

\begin{thm}
{\rm ~\cite{Rosenfeld,Boxer99}}
\label{continuityPreserveAdj}
A function $f:X\to Y$ is continuous if and only if $x \adj x'$ in $X$ 
implies $f(x) \adjeq f(x')$.
\end{thm}

Han's papers generally use the equivalent formulation that
$f$ is continuous if and only for every $x \in X$,
$f(N_{\kappa}(x)) \subset N_{\lambda}(f(x))$.

See also~\cite{Chen94,Chen04}, where similar notions are referred to as {\em immersions}, {\em gradually varied operators},
and {\em gradually varied mappings}.

A digital {\em isomorphism} (called {\em homeomorphism}
in~\cite{Boxer94}) is a $(\kappa,\lambda)$-continuous
surjection $f: X \to Y$ such that $f^{-1}: Y \to X$ is
$(\lambda,\kappa)$-continuous.

The literature uses {\em path} polymorphically: a
$(c_1,\kappa)$-continuous function $f: [0,m]_{\Z} \to X$
is a $\kappa$-path if $f([0,m]_{\Z})$ is a $\kappa$-path
as described above from $f(0)$ to $f(m)$.

\section{Han's variants on local isomorphisms}
The definition~\cite{Han05} of a digital covering map was simplified
to the following.

\begin{definition}
   {\rm \cite{BxAndWedges}}
   \label{coveringDef}
   Let $p: (E,\kappa) \to (B,\lambda)$ be a continuous
   surjection of digital images. The map $p$ is a
   {\em $(\kappa,\lambda)$ covering map} if and only if
   \begin{itemize}
       \item for every $b \in B$, there is an index set~$M$
             such that 
             \[p^{-1}(N_{\lambda}(b)) = 
             \bigcup_{i \in M} N_{\kappa}(e_i), 
             \mbox{ where } e_i \in p^{-1}(b);
             \]
       \item if $i,j \in M$, $i \neq j$, then
            $N_{\kappa}(e_i) \cap N_{\kappa}(e_j) =
            \emptyset$; and
       \item $p|_{N_{\kappa}(e_i)}: N_{\kappa}(e_i) \to
              N_{\lambda}(b)$ is a 
              $(\kappa,\lambda)$-isomorphism.
   \end{itemize}
\end{definition}

We find the following definition in Han's
paper~\cite{Han20} (not in~\cite{Han04} despite the claims
to the contrary in~\cite{Han20,HanEquiv}).

\begin{definition}
\label{PL-iso}
A digitally continuous map $h: (X,\kappa) \to (Y,\lambda)$
is a {\em pseudo-local (PL) isomorphism} if for every
$x \in X$, $h(N_{\kappa}(x)) \subset Y$ is
$\lambda$-isomorphic to 
$N_{\lambda}(h(x)) \subset Y$.
\end{definition}

In his paper~\cite{Han04}, Han gives the
following.

\begin{definition}
\label{localIso}
A digitally continuous map $h: (X,\kappa) \to (Y,\lambda)$
is a {\em local homeomorphism} [in more recent
terminology, a {\em local isomorphism}]
if for all $x \in X$, $h|_{N_{\kappa}(x)}$ is a 
$(\kappa,\lambda)$-homeomorphism
[$(\kappa,\lambda)$-isomorphism] onto
$N_{\lambda}(h(x))$.
\end{definition}

We have the following.

\begin{prop}
\label{localPLequiv}
Let $h: (X,\kappa) \to (Y,\lambda)$ be a 
digitally continuous map.
If $h$ is a local isomorphism 
then $h$ is a PL isomorphism.
\end{prop}

\begin{proof}
    Elementary and left to the
    reader.
\end{proof}

\begin{thm}
    {\rm (\cite{PakZak}, correcting an error of~\cite{Han04})}
    \label{PakZakEquiv}
    Let $f: (X,\kappa) \to (Y,\lambda)$ be a continuous
    surjection. Then $f$ is a digital covering map
    if and only if $f$ is a local isomorphism.
\end{thm}

We will also discuss the following notion.

\begin{definition}
    {\rm \cite{HanUnique}}
    \label{WL-iso}
    A function $h: (X,\kappa) \to (Y,\lambda)$
    is a {\em weakly local (WL) isomorphism} if for all
    $x \in X$, $h|_{N(x,1)}$ is an isomorphism onto
    $h(N(x,1))$.
\end{definition}

\section{Theorem 3.15(3) of~\cite{HanEquiv}}
\label{unmurk}
Part (3) of Theorem 3.15 of~\cite{HanEquiv}
states that
\begin{quote}
    Neither of a PL-$(k_0, k_1)$-isomorphism 
    and a WL-$(k_0, k_1)$-isomorphism 
    implies the other.
\end{quote}
The assertion is correct, but Han's argument
for the existence of $(X,k_0)$, $(Y,k_1)$,
and a WL-$(k_0, k_1)$-isomorphism
$h: X \to Y$ that is not a 
PL-$(k_0, k_1)$-isomorphism, is not as clear
as it could be. In the following, we 
clarify Han's argument.

In his example, Han makes use of 
an unstated assumption, namely
that $(Y,k_1)$ is connected.
He also assumes $k_0 = k_1 = \kappa$, that
\begin{equation}
\label{properSubset}
    X \subset Y \mbox{ but } X \neq Y
\end{equation}
and that 
$h$ is the inclusion map, trivially
a WL-$(\kappa, \kappa)$-isomorphism.

Note that since $(Y,\kappa)$ is connected,
~(\ref{properSubset}) implies there is a
$\kappa$-path $\{y,y'\} \subset Y$
such that $y \in X$, $y' \in Y \setminus X$.
Therefore,
\[ h(N(X,y,\kappa)) = N(X,y,\kappa) 
\mbox{ is a proper subset of }
  N(X,y,\kappa) \cup \{y'\} \subset N(Y,h(y),\kappa).
\]
Hence $h$ is not a 
PL-$(\kappa, \kappa)$-isomorphism.

\section{Theorem 3.20 of~\cite{HanEquiv}}
Let $(X,\kappa)$ and $(Y,\lambda)$ be digital
simple closed curves of $\ell_1$ and $\ell_2$
points, respectively.
Theorem~3.20 of~\cite{HanEquiv} states that 
$(X,\kappa)$ embeds into $(Y,\lambda)$ 
if and only if $\ell_1=\ell_2$.
Since a connected nonempty subset of
$(Y,\lambda)$ is either $(Y,\lambda)$ itself
or is isomorphic to a digital interval - which
is not even of the same digital homotopy
type as $(X,\kappa)$ - Han's assertion is an
easy consequence of the much older Theorem~5.1
of~\cite{BxHtpProps}, which 
states that $(X,\kappa)$ and $(Y,\lambda)$
have the same digital 
homotopy type if and only if 
$\ell_1=\ell_2$.

\section{Han's pseudo-covering maps in~~\cite{HanEquiv}}
Han defines a digital pseudo-covering
as follows.

\begin{definition}
\label{HanPseudocover}
{\rm \cite{HanUnique}}
    Let $p: (E,\kappa) \to (B,\lambda)$ be a surjection
    such that for every $b \in B$,
    \begin{enumerate}
        \item there is an index set $M$ such that
              $p^{-1} (N_{\kappa}(b,1)) = 
              \bigcup_{i \in M} N_{\kappa}(e_i,1)$, where
              $e_i \in p^{-1}(b)$;
        \item if $i,j \in M$ and $i \neq j$, then
              $N_{\kappa}(e_i,1) \cap N_{\kappa}(e_j,1) =
              \emptyset$; and
        \item $p|_{N_{\kappa}(e_i,1)}: N_{\kappa}(e_i,1) \to
               N_{\lambda}(b,1)$ is a WL-isomorphism for 
               all $i \in M$.
    \end{enumerate}
    Then $p$ is a {\em pseudo-covering} map.
\end{definition}

However, A. Pakdaman shows
in~\cite{Pak22} that Han's definition
does not effectively give us a new object of study.
In particular, Pakdaman shows the following.

\begin{thm}
\label{pseudoCoverIsCover}
    A digital pseudo-covering map as defined in
    Definition~\ref{HanPseudocover} is in fact a digital covering map.
\end{thm}

\begin{definition}
    {\rm \cite{Han05}} Let $p: (E,\kappa) \to (B,\lambda)$
    be $(\kappa,\lambda)$-continuous. Let 
    $f: [0,m]_{\Z} \to B$ be $(c_1, \lambda)$-continuous.
    A $(c_1,\kappa)$-continuous function 
    $\tilde{f}: [0,m]_{\Z} \to E$ such that
    $p \circ \tilde{f} = f$ is a 
    {\em (digital) path lifting of $f$}. If for every
    $b_0 \in B$, every $e_0 \in p^{-1}(b_0)$, and every
    path $f$ such that $f(0)=b_0$,
    \[ \mbox{there is a unique
    lifting $\tilde{f}$ such that $\tilde{f}(0)=e_0$,}
    \]
    then $p$ has the {\em unique path lifting property}.
\end{definition}

\begin{thm}
    \label{liftingThm}
    {\rm \cite{Han05}}
    A digital covering map has the unique path
    lifting property.
\end{thm}

Next, we show that several of the variants of
covering maps that we have discussed are
equivalent.

\begin{thm}
\label{coverEquivs}
 Let $p: (X,\kappa) \to (Y,\lambda)$ be a continuous
    surjection. Then the following are equivalent.
    \begin{enumerate}
        \item $p$ is a digital covering map.
        \item $p$ is a local isomorphism.
        \item $p$ is a pseudo-covering in the sense of
            Definition~\ref{HanPseudocover}.
        \item $p$ is a WL-isomorphism with the unique
              path lifting property.
\end{enumerate}
\end{thm}

\begin{proof}
    That 1) and 2) are equivalent is stated in 
    Theorem~\ref{PakZakEquiv}.

    That 1) implies 3) follows from 
    Definitions~\ref{coveringDef}
    and~\ref{HanPseudocover}.
    
    That 3) implies 1) is stated in
    Theorem~\ref{pseudoCoverIsCover}.

    It follows from Theorem~\ref{liftingThm}
    that 1) implies 4).
    
    To show 4) implies 2):
    suppose $p$ is a WL-isomorphism with the unique path 
    lifting property. Let
    $x \in X$, $p(x)=y \in Y$, 
    $y' \in N_{\lambda}(y) \setminus \{y\}$. Then
    $\{y,y'\}$ is a path in $(Y,\lambda)$, hence
    lifts to a unique path $\{x,x'\}$ in $(X,\kappa)$ with
    $p(x)=y$, $p(x')=y'$. Thus
    $N_{\lambda}(y) \subset p(N_{\kappa}(x))$. 
    Since continuity implies 
    $ p(N_{\kappa}(x)) \subset N_{\lambda}(y)$, we have
    $ p(N_{\kappa}(x)) = N_{\lambda}(y)$. Since $p$ is a
    WL-isomorphism, we have that $p$ is a local isomorphism.
\end{proof}

\begin{remark}
\label{badExl}
Han's Example 4.3(4) of~\cite{HanUnique} considers
(please note here ``$c_1$" is the label of
a point, so we will avoid using this notation
for 4-adjacency)
$C= \{ \, c_i \, \}_{i=0}^3$,
$D= \{ \, d_i \, \}_{i=0}^3$, where
\[ (0,0) = c_0 = d_0,~~~(1,0) = c_1 = d_1,~~~
    (1,1) = c_2 = d_3,~~~(0,1) = c_3=d_2,
\]
See Figure~\ref{fig:HanFig2}.
Han's claim, that $p$ is not a 
pseudo-4-covering (a pseudo-covering when 
$A$ and $B$ both use 4-adjacency), is correct, 
but this example should not have been
considered since $p$ is 
not 4-continuous: 
\[ c_0 \adj_4 c_3~~~ \mbox{ but }~~~
   p(c_0) = d_0 \not \adj_4 d_3 = p(c_3).\]

Since $p$ is (4,8)-continuous, perhaps Han intended to show that $p$ is not a 
(4,8)-pseudocovering as defined 
at Definition~\ref{HanPseudocover}.
This can be done by observing that 
\[ \#N_4(c_0)=3 \neq 4 = \#N_8(d_0) = \#N_8(p(c_0)). 
\]
Therefore, $p$ is not a (4,8)-local isomorphism, so
by Theorem~\ref{coverEquivs} is not a 
(4,8)-pseudocovering as
defined at Definition~\ref{HanPseudocover}.
\end{remark}

In the first paragraph of page 5104
of~\cite{HanEquiv}, Han attributes the definition
of a digital pointed continuous function to his
paper~\cite{Han05}. The definition should be attributed
to the earlier paper~\cite{Boxer99}.

   \begin{figure}
    \centering
\includegraphics{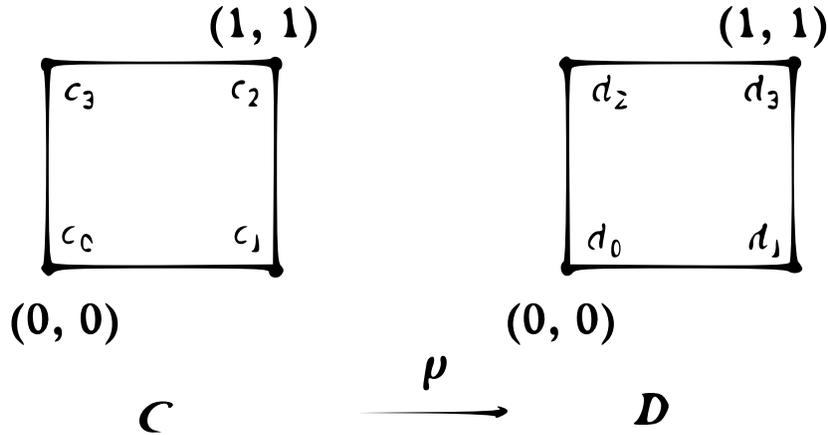}
    \caption{The function
    $p(c_i)=d_i$ of
    Han's Example 4.3(4)
    of~\cite{HanUnique};
    discussed here in
    Remark~\ref{badExl}
    }
    \label{fig:HanFig2}
\end{figure}

Pakdaman modifies Han's Definition~\ref{HanPseudocover}
as follows.

\begin{definition}
    \label{PakPseudocovering}
    {\rm \cite{Pak22}}
    Let $p: (E,\kappa) \to (B,\lambda)$ be a surjection
    of digital images. Suppose for all $b \in B$ we
    have the following.
    \begin{enumerate}
        \item for some index set $M$,
              $\bigcup_{i \in M} N_{\kappa}(e_i,1)
              \subset p^{-1}(N_{\lambda}(b,1))$ where
              $e_i \in p^{-1}(b)$;
        \item if $i,j \in M$ and $i \neq j$ then
              $N_{\kappa}(e_i,1) \cap N_{\kappa}(e_i,1) =
              \emptyset$; and
        \item for all $i \in M$,
              $p|{N_{\kappa}(e_i,1)}: N_{\kappa}(e_i,1) \to
               p(N_{\kappa}(e_i,1))$ is a
               $(\kappa, \lambda)$-isomorphism.
    \end{enumerate}
    Then $p$ is a {\em $(\kappa,\lambda)$-pseudocovering map}.
\end{definition}

Pakdaman proceeds to compare unique path lifting results
for pseudocovering maps based on
Definition~\ref{PakPseudocovering} with those
asserted by Han in~\cite{HanEquiv} based on
Definition~\ref{HanPseudocover}. He showed that
Definition~\ref{PakPseudocovering} gives something
not equivalent to a covering map, since such a
pseudocovering need not have the unique path lifting
property.

\section{Further remarks}
We have discussed various flaws in Han's
paper~\cite{HanEquiv}. We have shown 
that several variants
of digital covering maps that were presented
in~\cite{HanEquiv} are in fact equivalent.

Corrections and suggestions
of an anonymous referee are
acknowledged with gratitude.

\end{document}